\documentclass[12pt]{amsart}

\usepackage{amssymb}
\usepackage{bm}
\usepackage{graphicx}
\usepackage{psfrag}
\usepackage{hyperref}
\usepackage{amscd}
\usepackage{color}
\usepackage{url}
\usepackage{algpseudocode}
\usepackage{fancyhdr}
\usepackage{xy}
\input xy
\xyoption{all}

\newtheorem*{maintheorem*}{Main Theorem}
\newtheorem{theorem}{Theorem}[section]
\newtheorem{prop}[theorem]{Proposition}
\newtheorem{lemma}[theorem]{Lemma}
\newtheorem{cor}[theorem]{Corollary}
\newtheorem{conj}[theorem]{Conjecture}

\theoremstyle{definition}
\newtheorem{definition}[theorem]{Definition}
\newtheorem{remark}[theorem]{Remark}
\newtheorem{example}[theorem]{Example}

\numberwithin{equation}{section}

\newcommand{\nn}{\mathbb{N}}
\newcommand{\pp}{\mathbb{P}}
\newcommand{\qq}{\mathbb{Q}}
\newcommand{\rr}{\mathbb{R}}

\newcommand\pval{\mathsf{v}_p}
\newcommand\gp{\mathsf{gp}}
\newcommand{\supp}{\mathsf{Supp}}
\newcommand\lcm{\text{lcm}}

\hypersetup{
	pdftitle={Cones of Submonoids of Commutative Free Monoids},
	colorlinks=true,
	linkcolor=blue,
	citecolor=cyan,
	urlcolor=wine
}


\begin{document}
	
	\mbox{}
	\title[On the Molecules of Puiseux Monoids]{On the molecules of numerical semigroups, \\ Puiseux monoids, and Puiseux algebras}
	
	\author{Felix Gotti}
	\address{Department of Mathematics\\UC Berkeley\\Berkeley, CA 94720 \newline \indent Department of Mathematics\\Harvard University\\Cambridge, MA 02138}
	\email{felixgotti@berkeley.edu}
	\email{felixgotti@harvard.edu}
	
	\author{Marly Gotti}
	\address{Department of Mathematics\\University of Florida\\Gainesville, FL 32611}
	\email{marlycormar@ufl.edu}
	
	
	\date{\today}
	
	\begin{abstract}
		A \emph{molecule} is a nonzero non-unit element of an integral domain (resp., commutative cancellative monoid) having a unique factorization into irreducibles (resp., atoms). Here we study the molecules of Puiseux monoids as well as the molecules of their corresponding semigroup algebras, which we call Puiseux algebras. We begin by presenting, in the context of numerical semigroups, some results on the possible cardinalities of the sets of molecules and the sets of reducible molecules (i.e., molecules that are not irreducibles/atoms). Then we study the molecules in the more general context of Puiseux monoids. We construct infinitely many non-isomorphic atomic Puiseux monoids all whose molecules are atoms. In addition, we characterize the molecules of Puiseux monoids generated by rationals with prime denominators. Finally, we turn to investigate the molecules of Puiseux algebras. We provide a characterization of the molecules of the Puiseux algebras corresponding to root-closed Puiseux monoids. Then we use such a characterization to find an infinite class of Puiseux algebras with infinitely many non-associated reducible molecules.
	\end{abstract}
	
	\maketitle
	
	
\section{Introduction}
\label{sec:intro}

Let $M$ be a commutative cancellative monoid. A non-invertible element of $M$ is called an \emph{atom} if it cannot be expressed as a product of two non-invertible elements. 
If $x \in M$ can be expressed as a formal product of atoms, then such a formal product (up to associate and permutation) is called a \emph{factorization} of $x$. If every non-invertible element of $M$ has a factorization, then $M$ is called \emph{atomic}. Furthermore, the atoms and factorizations of an integral domain are the irreducible elements and the formal products of irreducible elements, respectively. All the undefined or informally-defined terms mentioned in this section will be formally introduced later on.

The elements having exactly one factorization are crucial in the study of factorization theory of commutative cancellative monoids and integral domains. Aiming to avoid repeated long descriptions, we call such elements \emph{molecules}. Molecules were first studied in the context of algebraic number theory by W.~Narkiewicz and other authors in the 1960's. For instance, in~\cite{wN66} and~\cite{wN66a} Narkiewicz studied some distributional aspects of the molecules of quadratic number fields. In addition, he gave an asymptotic formula for the number of (non-associated) integer molecules of any algebraic number field~\cite{wN72}. In this paper, we study the molecules of submonoids of~$(\qq_{\ge 0},+)$, including numerical semigroups, and the molecules of their corresponding semigroup algebras.

A \emph{numerical semigroup} is a cofinite submonoid of $(\nn_0,+)$, where $\nn_0 = \{0,1,2,\dots\}$. Every numerical semigroup is finitely generated by its set of atoms and, in particular, atomic. In addition, if $N \neq \nn_0$ is a numerical semigroup, then it contains only finitely many molecules. Notice, however, that every positive integer is a molecule of $(\nn_0,+)$. Figure~\ref{fig:molecule of four NS} shows the distribution of the sets of molecules of four numerical semigroups. We begin Section~\ref{sec:molecules of NS} pointing out how the molecules of numerical semigroups are related to the Betti elements. Then we show that each element in the set $\nn_{\ge 4} \cup \{\infty\}$ (and only such elements) can be the number of molecules of a numerical semigroup. We conclude our study of molecules of numerical semigroups exploring the possible cardinalities of the sets of reducible molecules (i.e., molecules that are not atoms).

A submonoid of $(\qq_{\ge 0},+)$ is called a \emph{Puiseux monoid}. Puiseux monoids were first studied in \cite{fG17} and have been systematically investigated since then (see~\cite{CGG20} and references therein). Albeit a natural generalization of the class of numerical semigroups, the class of Puiseux monoids contains members having infinitely many atoms and, consequently, infinitely many molecules. A Puiseux monoid is \emph{prime reciprocal} if it can be generated by rationals of the form $a/p$, where $p$ is a prime and $a$ is a positive integer not divisible by $p$.  
In Section~\ref{sec:molecules of PM}, we study the sets of molecules of Puiseux monoids, finding infinitely many non-isomorphic Puiseux monoids all whose molecules are atoms (in contrast to the fact that the set of molecules of a numerical semigroup always differs from its set of atoms). In addition, we construct infinitely many non-isomorphic Puiseux monoids having infinitely many molecules that are not atoms (in contrast to the fact that the set of molecules of a nontrivial numerical semigroup is always finite). We conclude Section~\ref{sec:molecules of PM} characterizing the sets of molecules of prime reciprocal Puiseux monoids.
\begin{figure}
	\centering
	\includegraphics[width = 11cm]{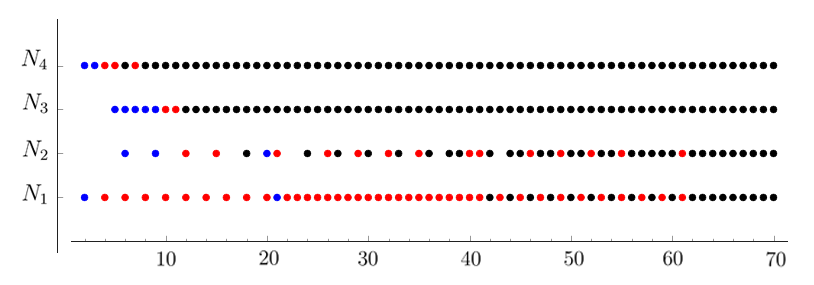}
	\caption{The dots on the horizontal line labeled by $N_i$ represent the nonzero elements of the numerical semigroup $N_i$; here we are setting $N_1 = \langle 2, 21 \rangle$, $N_2 = \langle 6,9,20 \rangle$, $N_3 = \langle 5,6,7,8,9 \rangle$, and $N_4 = \langle 2,3 \rangle$.
		Atoms are represented in blue, molecules that are not atoms in red, and non-molecules in black.}
	\label{fig:molecule of four NS}
\end{figure}

The final section of this paper is dedicated to the molecules of the semigroup algebras of Puiseux monoids, which we call \emph{Puiseux algebras}. Puiseux algebras have been recently studied in~\cite{ACHZ07,CG19,GK20,fG18}. First, for a fixed field $F$ we establish a bijection between the set molecules of a Puiseux monoid and the set of non-associated monomial molecules of its corresponding Puiseux algebra over $F$. Then we characterize the molecules of Puiseux algebras of root-closed Puiseux monoids. We conclude this paper using the previous characterization to exhibit a class of Puiseux algebras having infinitely many molecules that are neither monomials nor irreducibles.
\medskip

\section{Monoids, Atoms, and Molecules}
\label{sec:background}
\smallskip

\subsection{General Notation} In this section we review the nomenclature and main concepts on commutative monoids and factorization theory we shall be using later. For a self-contained approach to the theory of commutative monoids we suggest~\cite{pG01} by P.~A.~Grillet, and for background on non-unique factorization theory of atomic monoids and integral domains the reader might want to consult \cite{GH06}~by A.~Geroldinger and F.~Halter-Koch.

We use the double-struck symbols $\mathbb{N}$ and $\mathbb{N}_0$ to denote the set of positive integers and the set of nonnegative integers, respectively, while we let $\pp$ denote the set of primes. If $R \subseteq \rr$ and $r \in \mathbb{R}$, then we set
\[
	R_{\ge r} := \{x \in R : x \ge r\}.
\]
The notation $R_{> r}$ is used in a similar way. We let the symbol~$\emptyset$ denote the empty set. If $q \in \qq_{> 0}$, then the unique $a,b \in \nn$ such that $q = a/b$ and $\gcd(a,b)=1$ are denoted by $\mathsf{n}(q)$ and $\mathsf{d}(q)$, respectively. For $Q \subseteq \qq_{>0}$, we call
\[
	\mathsf{n}(Q) := \{\mathsf{n}(q) : q \in Q\} \quad \text{and} \quad \mathsf{d}(Q) := \{\mathsf{d}(q) : q \in Q\}
\]
the \emph{numerator set} and \emph{denominator set} of $Q$, respectively. In addition, if $S$ is a set consisting of primes and $q \in \qq_{> 0}$, then we set
\[
	\mathsf{D}_S(q) := \{ p \in S : p \mid \mathsf{d}(q) \} \quad \text{and} \quad \mathsf{D}_S(Q) := \cup_{q \in Q} \mathsf{D}_S(q).
\]
For $p \in \pp$, the $p$-\emph{adic valuation} on $\qq_{\ge 0}$ is the map defined by $\pval(q) = \pval(\mathsf{n}(q)) - \pval(\mathsf{d}(q))$ for $q \in \mathbb{Q}_{> 0}$ and $\pval(0) = \infty$, where for $n \in \nn$ the value $\pval(n)$ is the exponent of the maximal power of $p$ dividing $n$. It can be easily seen that the $p$-adic valuation satisfies that $\pval(q_1 + \dots + q_n) \ge \min\{\pval(q_1), \dots, \pval(q_n) \}$ for every $n \in \nn$ and $q_1, \dots, q_n \in \qq_{> 0}$.
\medskip

\subsection{Monoids} Throughout this paper, we will tacitly assume that the term \emph{monoid} by itself always refers to a commutative and cancellative semigroup with identity. In addition, we will use additive notation by default and switch to multiplicative notation only when necessary (in which case, the notation will be clear from the context). For a monoid $M$, we let $M^\bullet$ denote the set $M \! \setminus \! \{0\}$. If $a,c \in M$, then we say that~$a$ \emph{divides}~$c$ \emph{in} $M$ and write $a \mid_M c$ provided that $c = a + b$ for some $b \in M$. We write $M = \langle S \rangle$ when $M$ is generated by a set $S$. The monoid $M$ is \emph{finitely generated} if it can be generated by a finite set; otherwise, $M$ is said to be \emph{non-finitely generated}. A succinct exposition of finitely generated monoids can be found in~\cite{GR99}.
\medskip

\subsection{Atoms and Molecules} The set of invertible elements of $M$ is denoted by $M^\times$\!, and $M$ is said to be \emph{reduced} if $M^\times$ contains only the identity element.

\begin{definition}
	An element $a \in M \! \setminus \! M^\times$ is an \emph{atom} provided that for all $u,v \in M$ the fact that $a = u + v$ implies that either $u \in M^\times$ or $v \in M^\times$. The set of atoms of $M$ is denoted by $\mathcal{A}(M)$, and $M$ is called \emph{atomic} if $M = \langle \mathcal{A}(M) \rangle$.
\end{definition}

Let $M$ be a reduced monoid. Then the \emph{factorization monoid} $\mathsf{Z}(M)$ of $M$ is the free commutative monoid on $\mathcal{A}(M)$. The elements of $\mathsf{Z}(M)$, which are formal sums of atoms, are called \emph{factorizations}. If $z = a_1 + \dots + a_n \in \mathsf{Z}(M)$ for some $a_1, \dots, a_n \in \mathcal{A}(M)$, then $|z| := n$ is called the \emph{length} of $z$. As $\mathsf{Z}(M)$ is free on $\mathcal{A}(M)$, there is a unique monoid homomorphism from $\mathsf{Z}(M)$ to $M$ determined by the assignment $a \mapsto a$ for all $a \in \mathcal{A}(M)$. Such a monoid homomorphism is called the \emph{factorization homomorphism} of $M$ and is denoted by $\phi_M$ (or just $\phi$ when there is no risk of ambiguity involved). For $x \in M$, the sets
\[
	\mathsf{Z}(x) := \mathsf{Z}_{M}(x) := \phi^{-1}(x) \subseteq \mathsf{Z}(M) \quad \text{ and } \quad \mathsf{L}(x) := \mathsf{L}_{M}(x) := \{|z| : z \in \mathsf{Z}(x)\}
\]
are called the \emph{set of factorizations} and the \emph{set of lengths} of $x$, respectively. Clearly, $M$ is atomic if and only if $\mathsf{Z}(x) \neq \emptyset$ for all $x \in M$.

Let $M_{\text{red}}$ denote the set of classes of $M$ under the equivalence relation $x \sim y$ if $y = x + u$ for some $u \in M^\times$. It turns out that $M_{\text{red}}$ is a monoid with the addition operation inherited from $M$. The monoid $M_{\text{red}}$ is called the \emph{reduced monoid} of $M$ (clearly, $M_{\text{red}}$ is reduced). Note that an element $a$ belongs to $\mathcal{A}(M)$ if and only if the class of $a$ belongs to $\mathcal{A}(M_{\text{red}})$. If $M$ is an atomic monoid (that is not necessarily reduced), then we set $\mathsf{Z}(M) := \mathsf{Z}(M_{\text{red}})$ and, for $x \in M$, we define $\mathsf{Z}(x)$ and $\mathsf{L}(x)$ in terms of $\mathsf{Z}(M)$ as we did for the reduced case.

As one of the main purposes of this paper is to study elements with exactly one factorization in Puiseux monoids (in particular, numerical semigroups), we introduce the following definition.

\begin{definition}
	Let $M$ be a monoid. We say that an element $m \in M \setminus M^\times$ is a \emph{molecule} provided that $|\mathsf{Z}(m)| = 1$. The set of all molecules of $M$ is denoted by $\mathcal{M}(M)$.
\end{definition}

It is clear that the set of atoms of any monoid is contained in the set of molecules. However, such an inclusion might be proper (consider, for instance, the additive monoid $\nn_0$). In addition, for any atomic monoid $M$ the set $\mathcal{M}(M)$ is \emph{divisor-closed} in the sense that if $m \in \mathcal{M}(M)$ and $m' \mid_M m$ for some $m' \in M \setminus M^\times$, then $m' \in \mathcal{M}(M)$. If the condition of atomicity is dropped, then this observation is not necessarily true (see Example~\ref{ex:PM whose set of molecules is not divisor-closed}).
\medskip

\section{Molecules of Numerical Semigroups}
\label{sec:molecules of NS}

In this section we study the sets of molecules of numerical semigroups, putting particular emphasis on their possible cardinalities. 

\begin{definition}
	A \emph{numerical semigroup} is a cofinite additive submonoid of $\nn_0$.
\end{definition}

\noindent We let $\mathcal{N}$ denote the class consisting of all numerical semigroups (up to isomorphism). We say that $N \in \mathcal{N}$ is \emph{nontrivial} if $\nn_0 \! \setminus \! N$ is not empty, and we let $\mathcal{N}^\bullet$ denote the class of all nontrivial numerical semigroups. Every $N \in \mathcal{N}$ has a unique minimal set of generators $A$, which is finite. The cardinality of $A$ is called the \emph{embedding dimension} of $N$. Suppose that $N$ has embedding dimension $n$, and let $N = \langle a_1 , \dots, a_n \rangle$ (we always assume that $a_1 < \dots < a_n$). Then $\gcd(a_1, \dots, a_n) = 1$ and $\mathcal{A}(N) = \{a_1, \dots, a_n\}$. In particular, every numerical semigroup is atomic. When $N$ is nontrivial, the maximum of $\nn_0 \! \setminus \! N$ is called the \emph{Frobenius number} of $N$. Here we let $\mathsf{F}(N)$ denote the Frobenius number of~$N$. See \cite{GR09} for a friendly introduction to numerical semigroups.

\begin{example} \label{ex:NS with embedding dimension two}
	For $k \ge 1$, consider the numerical semigroup $N_1 = \langle 2, 21 \rangle$, whose molecules are depicted in Figure~\ref{fig:molecule of four NS}. It is not hard to see that $x \in N_1^\bullet$ is a molecule if and only if every factorization of $x$ contains at most one copy of $21$. Therefore
	\[
		\mathcal{M}(N_1) = \big\{2m + 21n : 0 \le m < 21, n \in \{0,1\}, \, \text{and} \, (m, n) \neq (0, 0)\big\}.
	\]
	In addition, if $2m + 21n = 2m' + 21n'$ for some $m,m' \in \{0,\dots, 20\}$ and $n,n' \in \{0,1\}$, then one can readily check that $m = m'$ and $n = n'$. Hence $|\mathcal{M}(N_1)| = 41$.
\end{example}
\smallskip

\subsection{Betty Elements} Let $N = \langle a_1, \dots, a_n \rangle$ be a minimally generated numerical semigroup. We always represent an element of $\mathsf{Z}(N)$ with an $n$-tuple $z = (c_1, \dots, c_n) \in \nn^n_0$, where the entry $c_i$ specifies the number of copies of $a_i$ that appear in $z$. Clearly, $|z| = c_1 + \dots + c_n$. Given factorizations $z = (c_1, \dots, c_n)$ and $z' = (c'_1, \dots, c'_n)$, we define
\[
	\gcd(z,z') = (\min\{c_1,c'_1\}, \dots, \min\{c_n,c'_n\}).
\]
The \emph{factorization graph} of $x \in N$, denoted by $\nabla\!_x(N)$ (or just $\nabla\!_x$ when no risk of confusion exists), is the graph with vertices $\mathsf{Z}(x)$ and edges between those $z, z' \in \mathsf{Z}(x)$ satisfying that $\gcd(z,z') \neq 0$. The element $x$ is called a \emph{Betti element} of $N$ provided that $\nabla\!_x$ is disconnected. The set of Betti elements of $N$ is denoted by $\text{Betti}(N)$.

\begin{example}
	Take $N$ to be the numerical semigroup $\langle 14, 16, 18, 21, 45 \rangle$. A computation in SAGE using the \texttt{numericalsgps GAP} package reveals that $N$ has nine Betti elements. In particular, $90 \in \text{Betti}(N)$. In Figure~\ref{fig:factorization graphs} one can see the disconnected factorization graph of the Betti element $90$ on the left and the connected factorization graph of the non-Betti element $84$ on the right. 
	\begin{figure}
		\centering
		\includegraphics[width = 5.4cm]{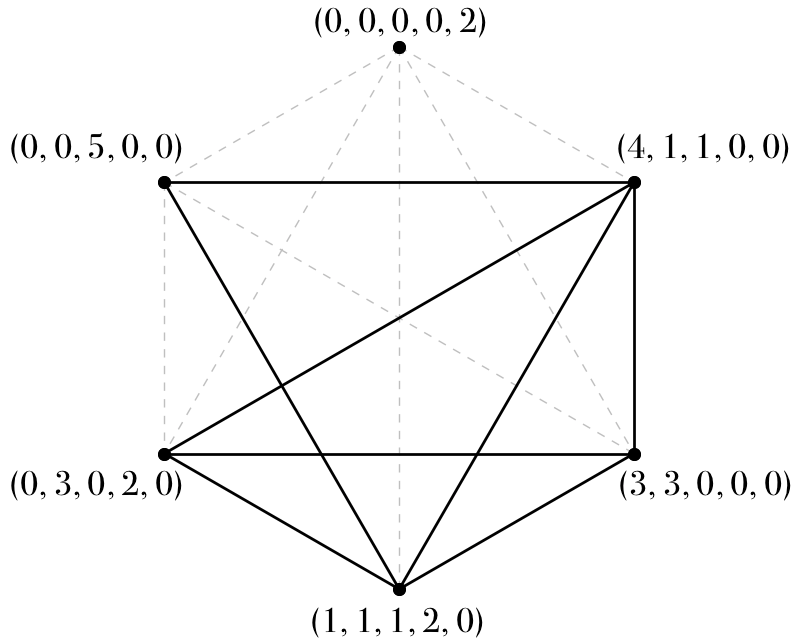}
		\hspace{8pt}
		\includegraphics[width = 5.4cm]{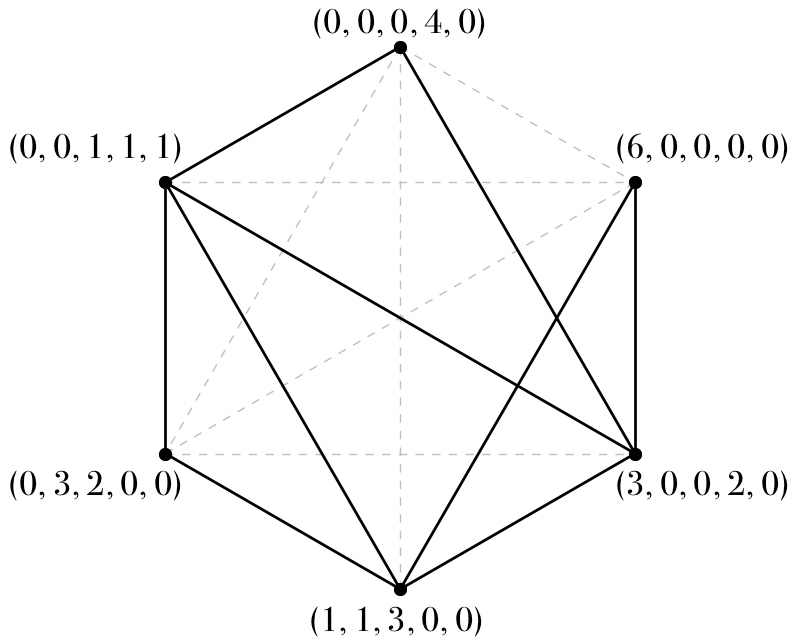}
		\caption{The factorization graphs of $90 \in \text{Betti}(N)$ (on the left) and $84 \notin \text{Betti}(N)$ (on the right), where $N$ is the numerical semigroup $\langle 14,16,18,21,45 \rangle$.}
		\label{fig:factorization graphs}
	\end{figure}
\end{example}

Observe that $0 \notin \text{Betti}(N)$ since $|\mathsf{Z}(0)| = 1$. It is well known that every numerical semigroup has finitely many Betti elements. Betti elements play a fundamental role in the study of uniquely-presented numerical semigroups \cite{GO10} and the study of \emph{delta sets} of BF-monoids \cite{CGLMS12}. For a more general notion of Betti element, meaning the \emph{syzygies} of an $\nn^n$-graded module, see~\cite{MS04}. In a numerical semigroup, Betti elements and molecules are closely related.

\begin{remark}
	Let $N$ be a numerical semigroup. An element $m \in N$ is a molecule if and only if $\beta \nmid_N m$ for any $\beta \in \text{Betti}(N)$.
\end{remark}

\begin{proof}
	For the direct implication, suppose that $m$ is a molecule of $N$ and take $\alpha \in N$ such that $\alpha \mid_N m$. As the set of molecules is closed under division, $|\mathsf{Z}(\alpha)| = 1$. This implies that $\nabla \!_\alpha$ is connected and, therefore, $\alpha$ cannot be a Betti element. The reverse implication is just a rephrasing of \cite[Lemma~1]{GO10}.
\end{proof}
\medskip

\subsection{On the Sizes of the Sets of Molecules} Obviously, for every $n \in \nn$ there exists a numerical semigroup having exactly $n$ atoms. The next proposition answers the same realization question replacing the concept of an atom by that of a molecule. Recall that $\mathcal{N}^\bullet$ denotes the class of all nontrivial numerical semigroups.

\begin{prop} \label{prop:realization for the size of the sets of molecules}
	$\{|\mathcal{M}(N)| : N \in \mathcal{N}^\bullet\} = \nn_{\ge 4}$.
\end{prop}

\begin{proof}
	Let $N$ be a nontrivial numerical semigroup. Then $N$ must contain at least two atoms. Let $a$ and $b$ denote the two smallest atoms of $N$, and assume that $a < b$. Note that $2a$ and $a+b$ are distinct molecules that are not atoms. Hence $|\mathcal{M}(N)| \ge 4$. As a result, $\{|\mathcal{M}(N)| : N \in \mathcal{N}^\bullet\} \subseteq \nn_{\ge 4} \cup \{ \infty \}$. Now take $x \in \nn$ with $x > \mathsf{F}(N) + a b$. Since $x' := x - a b > \mathsf{F}(N)$, we see that $x' \in N$ and, therefore, $\mathsf{Z}(x')$ contains at least one factorization, namely $z$. So we can find two distinct factorizations of $x$ by adding to~$z$ either $a$ copies of $b$ or $b$ copies of $a$. Then $\mathsf{F}(N) + a b$ is an upper bound for $\mathcal{M}(N)$, which means that $|\mathcal{M}(N)| \in \nn_{\ge 4}$. Thus, $\{|\mathcal{M}(N)| : N \in \mathcal{N}^\bullet\} \subseteq \nn_{\ge 4}$.
	
	To argue the reverse inclusion, suppose that $n \in \nn_{\ge 4}$, and let us find $N \in \mathcal{N}$ with $|\mathcal{M}(N)| = n$. For $n = 4$, we can take the numerical semigroup $\langle 2, 3 \rangle$ (see Figure~\ref{fig:molecule of four NS}). For $n > 4$, consider the numerical semigroup
	\[
		N = \langle n-2, n-1, \dots, 2(n-2)-1 \rangle.
	\]
	It follows immediately that $\mathcal{A}(N) = \{n-2, n-1, \dots, 2(n-2)-1\}$. In addition, it is not hard to see that $2(n-2), 2(n-2)+1 \in \mathcal{M}(N)$ while $k \notin \mathcal{M}(N)$ for any $k > 2(n-2)+1$. Consequently, $\mathcal{M}(N) = \mathcal{A}(N) \cup \{2(n-2), 2(n-2)+1\}$, which implies that $|\mathcal{M}(N)| = n$. Therefore $\{|\mathcal{M}(N)| : N \in \mathcal{N}\} \supseteq \nn_{\ge 4}$, which completes the proof.
\end{proof}

\begin{cor}
	The monoid $(\nn_0,+)$ is the only numerical semigroup having infinitely many molecules.
\end{cor}

In Proposition~\ref{prop:realization for the size of the sets of molecules} we have fully described the set $\{|\mathcal{M}(N)| : N \in \mathcal{N} \}$. A full description of the set $\{|\mathcal{M}(N) \setminus \mathcal{A}(N)| : N \in \mathcal{N} \}$ seems to be significantly more involved. However, the next theorem offers some evidence to believe that
\[
	\{|\mathcal{M}(N) \setminus \mathcal{A}(N)| : N \in \mathcal{N} \} = \nn_{\ge 2} \cup \{\infty\}.
\]

\begin{theorem} \label{thm:more molecules than atoms in any NS}
	The following statements hold.
	\begin{enumerate}
		\item $\{|\mathcal{M}(N) \setminus \mathcal{A}(N)| : N \in \mathcal{N}^\bullet \} \subseteq \nn_{\ge 2}$.
		\vspace{3pt}
		\item $|\mathcal{M}(N) \! \setminus \! \mathcal{A}(N)| = 2$ for infinitely many numerical semigroups $N$.
		\vspace{3pt}
		\item For each $k \in \nn$, there is a numerical semigroup $N$ with $|\mathcal{M}(N) \! \setminus \! \mathcal{A}(N)| > k$.
	\end{enumerate}
\end{theorem}

\begin{proof}
	To prove~(1), take $N \in \mathcal{N}^\bullet$. Then we can assume that $N$ has embedding dimension $n$ with $n \ge 2$. Take $a_1, \dots, a_n \in \nn$ such that $a_1 < \dots < a_n$ such that $N = \langle a_1, \dots, a_n \rangle$. Since $a_1 < a_2 < a_j$ for every $j = 3, \dots,n$, the elements $2a_1$ and $a_1 + a_2$ are two distinct molecules of $N$ that are not atoms. Hence $\mathcal{M}(N) \setminus \mathcal{A}(N) \subseteq \nn_{\ge 2} \cup \{\infty\}$. On the other hand, Proposition~\ref{prop:realization for the size of the sets of molecules} guarantees that $|\mathcal{M}(N)| < \infty$, which implies that $|\mathcal{M}(N) \setminus \mathcal{A}(N)| < \infty$. As a result, the statement~(1) follows.
	
	To verify the statement~(2), one only needs to consider for every $n \in \nn$ the numerical semigroup $N_n := \{0\} \cup \nn_{\ge n-2}$. The minimal set of generators of $N_n$ is the $(n-2)$-element set $\{n-2, n-1, \dots, 2(n-2)-1 \}$ and, as we have already argued in the proof of Proposition~\ref{prop:realization for the size of the sets of molecules}, the set $\mathcal{M}(N_n) \! \setminus \! \mathcal{A}(N_n)$ consists precisely of two elements.
	
	Finally, let us prove condition~(3). To do this, we first argue that for any $a,b \in \nn_{\ge 2}$ with $\gcd(a,b) = 1$ the numerical semigroup $\langle a, b \rangle$ has exactly $ab-1$ molecules (cf. Example~\ref{ex:NS with embedding dimension two}). Assume $a < b$, take $N := \langle a, b \rangle $, and set
	\[
		\mathcal{M} = \{ma + nb : 0 \le m < b, \, 0 \le n < a, \, \text{and} \ \, (m, n) \neq (0, 0)\}.
	\]
	Now take $x \in N$ to be a molecule of $N$. As $|\mathsf{Z}(x)| = 1$, the unique factorization $z := (c_1, c_2) \in \mathsf{Z}(x)$ (with $c_1,c_2 \in \nn_0$) satisfies that $c_1 < b$; otherwise, we could exchange~$b$ copies of the atom $a$ by $a$ copies of the atom $b$ to obtain another factorization of $x$. A similar argument ensures that $c_2 < a$. As a consequence, $\mathcal{M}(N) \subseteq \mathcal{M}$. On the other hand, if $ma + nb = m'a + n'b$ for some $m,m',n,n' \in \nn_0$, then $\gcd(a,b) = 1$ implies that $b \mid m-m'$ and $a \mid n-n'$. Because of this observation, the element $(b-1)a + (a-1)b$ has only the obvious factorization, namely $(b-1,a-1)$. Since $(b-1)a + (a-1)b$ is a molecule satisfying that $y \mid_N (b-1)a + (a-1)b$ for every $y \in \mathcal{M}$, the inclusion $\mathcal{M} \subseteq \mathcal{M}(N)$ holds. Hence $|\mathcal{M}(N)| = |\mathcal{M}| = ab-1$. To argue the statement~(3) now, it suffices to take $N := \langle 2, 2k+1 \rangle$.
\end{proof}
\medskip

We conclude this section with the following conjecture.

\begin{conj} \label{conj:molecules minus atoms in NS}
	For every $n \in \nn_{\ge 2}$, there exists a numerical semigroup $N$ such that $|\mathcal{M}(N) \! \setminus \! \mathcal{A}(N)| = n$.
\end{conj}
\medskip

\section{Molecules of Puiseux Monoids}
\label{sec:molecules of PM}
\smallskip

\subsection{Molecules of Generic Puiseux Monoids} In this section we study the sets of molecules of Puiseux monoids. We will argue that there are infinitely many non-finitely generated atomic Puiseux monoids~$P$ such that $|\mathcal{M}(P) \setminus \mathcal{A}(P)| = \infty$. On the other hand, we will prove that, unlike the case of numerical semigroups, there are infinitely many non-isomorphic atomic Puiseux monoids all whose molecules are, indeed, atoms. The last part of this section is dedicated to characterize the molecules of prime reciprocal Puiseux monoids.

\begin{definition}
	A \emph{Puiseux monoid} is an additive submonoid of $\qq_{\ge 0}$.
\end{definition}

\noindent Clearly, every numerical semigroup is naturally isomorphic to a Puiseux monoid. However, Puiseux monoids are not necessarily finitely generated or atomic, as the next example illustrates. The atomic structure of Puiseux monoids has been investigated recently~\cite{fG19,fG17,GG18}. At the end of Section~\ref{sec:background} we mentioned that the set of molecules of an atomic monoid is divisor-closed. The next example indicates that this property may not hold for non-atomic monoids.

\begin{example} \label{ex:PM whose set of molecules is not divisor-closed}
	Consider the Puiseux monoid
	\[
		P = \bigg\langle \frac 25, \frac 35, \frac 1{2^n} \, : \, n \in \nn \bigg\rangle.
	\]
	First, observe that $0$ is a limit point of $P^\bullet$, and so $P$ cannot be finitely generated. After a few easy verifications, one can see that $\mathcal{A}(P) = \{2/5, 3/5\}$. On the other hand, it is clear that $1/2 \notin \langle 2/5,3/5 \rangle$, so $P$ is not atomic. Observe now that~$\mathsf{Z}(1)$ contains only one factorization, namely $2/5 + 3/5$. Therefore $1 \in \mathcal{M}(P)$. Since $\mathsf{Z}(1/2)$ is empty, $1/2$ is not a molecule of $P$. However, $1/2 \mid_P 1$. As a result, $\mathcal{M}(P)$ is not divisor-closed.
\end{example}

Although the additive monoid $\nn_0$ contains only one atom, it has infinitely many molecules. The next result implies that $\nn_0$ is basically the only atomic Puiseux monoid having finitely many atoms and infinitely many molecules. 

\begin{prop} \label{prop:atoms-molecules cardinality for PM}
	Let $P$ be a Puiseux monoid. Then $|\mathcal{M}(P)| \in \nn_{\ge 2}$ if and only if $|\mathcal{A}(P)| \in \nn_{\ge 2}$. 
\end{prop}

\begin{proof}
	Suppose first that $|\mathcal{M}(P)| \in \nn_{\ge 2}$. As every atom is a molecule, $\mathcal{A}(P)$ is finite. Furthermore, note that if $\mathcal{A}(P) = \{a\}$, then every element of the set $S = \{na : n \in \nn\}$ would be a molecule, which is not possible as $|S| = \infty$. As a result, $|\mathcal{A}(P)| \in \nn_{\ge 2}$. Conversely, suppose that $|\mathcal{A}(P)| \in \nn_{\ge 2}$. Since the elements in $P \! \setminus \! \langle \mathcal{A}(P) \rangle$ have no factorizations, $\mathcal{M}(P) = \mathcal{M}(\langle \mathcal{A}(P) \rangle)$. Therefore there is no loss in assuming that $P$ is atomic. As $1 < |\mathcal{A}(P)| < \infty$, the monoid $P$ is isomorphic to a nontrivial numerical semigroup. The proposition now follows from the fact that nontrivial numerical semigroups have finitely many molecules.
\end{proof}

\begin{cor}
	If $P$ is a Puiseux monoid, then $|\mathcal{M}(P)| \neq 1$.
\end{cor}

The set of atoms of a numerical semigroup is always strictly contained in its set of molecules. However, there are many atomic Puiseux monoids which do not satisfy such a property. Before proceeding to formalize this observation, let us mention that if two Puiseux monoids $P$ and $P'$ are isomorphic, then there exists $q \in \qq_{>0}$ such that $P' = qP$; this is a consequence of~\cite[Proposition~3.2(1)]{fG18a}.

\begin{theorem}[cf. Theorem~\ref{thm:more molecules than atoms in any NS}(1)] \label{thm:F(M) equals A(M)}
	There are infinitely many non-isomorphic atomic Puiseux monoids $P$ satisfying that $\mathcal{M}(P) = \mathcal{A}(P)$.
\end{theorem}

\begin{proof}
	Let $\mathcal{S} = \{S_n : n \in \nn\}$ be a collection of infinite and pairwise-disjoint sets of primes. Now take $S = S_n$ for some arbitrary $n \in \nn$, and label the primes in $S$ strictly increasingly by $p_1, p_2, \dots$. Recall that $\mathsf{D}_S(r)$ denotes the set of primes in $S$ dividing $\mathsf{d}(r)$ and that $\mathsf{D}_S(R) = \cup_{r \in R} \mathsf{D}_S(r)$ for $R \subseteq \qq_{> 0}$. We proceed to construct a Puiseux monoid $P_S$ satisfying that $\mathsf{D}_S(P_S) = S$.
	
	Take $P_1 := \langle 1/p_1 \rangle$ and $P_2 := \langle P_1, 2/(p_1p_2) \rangle$. In general, suppose that $P_k$ is a finitely generated Puiseux monoid such that $\mathsf{D}_S(P_k) \subset S$, and let $r_1, \dots, r_{n_k}$ be all the elements in $P_k$ which can be written as a sum of two atoms. Clearly, $n_k \ge 1$. Because $|S| = \infty$, one can take $p'_1, \dots, p'_{n_k}$ to be primes in $S \! \setminus \! \mathsf{D}_S(P_k)$ satisfying that $p'_i \nmid \mathsf{n}(r_i)$. Now consider the following finitely generated Puiseux monoid
	\[
		P_{k+1} := \bigg\langle P_k \cup \bigg\{ \frac{r_1}{p'_1}, \dots, \frac{r_{n_k}}{p'_{n_k}} \bigg\} \bigg\rangle.
	\]
	For every $i \in \{1,\dots,n_k\}$, there is only one element in $P_k \cup \{r_1/p'_1, \dots, r_{n_k}/p'_{n_k}\}$ whose denominator is divisible by $p'_i$, namely $r_i/p'_i$. Therefore $r_i/p'_i \in \mathcal{A}(P_{k+1})$ for $i=1, \dots, n_k$. To check that $\mathcal{A}(P_k) \subset \mathcal{A}(P_{k+1})$, fix $a \in \mathcal{A}(P_k)$ and take
	\begin{equation} \label{eq:FM equals AM}
		z := \sum_{i=1}^m \alpha_i a_i + \sum_{i=1}^{n_k} \beta_i \frac{r_i}{p'_i} \in \mathsf{Z}_{P_{k+1}}(a),
	\end{equation}
	where $a_1, \dots, a_m$ are pairwise distinct atoms in $\mathcal{A}(P_{k+1}) \cap P_k$ and $\alpha_i, \beta_j$ are nonnegative coefficients for $i = 1,\dots,m$ and $j = 1,\dots, n_k$. In particular, $a_1, \dots, a_m \in \mathcal{A}(P_k)$. For each $i=1,\dots,n_k$, the fact that the $p'_i$-adic valuation of $a$ is nonnegative implies that $p'_i \mid \beta_i$. Hence
	\[
		a = \sum_{i=1}^m \alpha_i a_i + \sum_{i=1}^{n_k} \beta'_i r_i,
	\]
	where $\beta'_i = \beta_i/p'_i \in \nn_0$. Since $r_i \in \mathcal{A}(P_k) + \mathcal{A}(P_k)$ and $(\beta_i/p'_i)r_i \mid_{P_k} a$ for every $i=1,\dots,n_k$, one obtains that $\beta_1 = \dots = \beta_{n_k} = 0$. As a result, $a = \sum_{i=1}^m \alpha_i a_i$. Because $a \in \mathcal{A}(P_k)$, the factorization $\sum_{i=1}^m \alpha_i a_i$ in $\mathsf{Z}_{P_k}(a)$ must have length $1$, i.e, $\sum_{i=1}^m \alpha_i = 1$. Thus, $\sum_{i=1}^m \alpha_i + \sum_{i=1}^{n_k} \beta_i =~1$, which means that $z$ has length $1$ and so $a \in \mathcal{A}(P_{k+1})$. As a result, the inclusion $\mathcal{A}(P_k) \subseteq \mathcal{A}(P_{k+1})$ holds. Observe that because $n_k \ge 1$, the previous containment must be strict. Now set
	\[
		P_S = \bigcup_{k \in \nn} P_k.
	\]
	Let us verify that $P_S$ is an atomic monoid satisfying that $\mathcal{A}(P_S) = \cup_{k \in \nn} \mathcal{A}(P_k)$. Since $P_k$ is atomic for every $k \in \nn$, the inclusion chain $\mathcal{A}(P_1) \subset \mathcal{A}(P_2) \subset \dots$ implies that $P_1 \subset P_2 \subset \dots$. In addition, if $a_0 = a_1 + \dots + a_m$ for $m \in \nn$ and $a_0, a_1, \dots, a_m \in P_S$, then $a_0 = a_1 + \dots + a_m$ will also hold in $P_k$ for some $k \in \nn$ large enough. This immediately implies that $\cup_{k \in \nn} \mathcal{A}(P_k) \subseteq \mathcal{A}(P_S)$. Since the reverse inclusion follows trivially, $\mathcal{A}(P_S) = \cup_{k \in \nn} \mathcal{A}(P_k)$. To check that $P_S$ is atomic, take $x \in P_S^\bullet$. Then there exists $k \in \nn$ such that $x \in P_k$ and, because $P_k$ is atomic, $x \in \langle \mathcal{A}(P_k) \rangle \subseteq \langle \mathcal{A}(P_S) \rangle$. Hence $P_S$ is atomic.
	
	To check that $\mathcal{M}(P_S) = \mathcal{A}(P_S)$, suppose that $m$ is a molecule of $P_S$, and then take $K \in \nn$ such that $m \in P_k$ for every $k \ge K$. Since $\mathcal{A}(P_k) \subset \mathcal{A}(P_{k+1}) \subset \dots$, we obtain that $\mathsf{Z}_{P_k}(m) \subseteq \mathsf{Z}_{P_{k+1}}(m) \subseteq \dots$. Moreover, $\cup_{k \ge K} \mathcal{A}(P_k) = \mathcal{A}(P_S)$ implies that $\cup_{k \ge K} \mathsf{Z}_{P_k}(m) = \mathsf{Z}_{P_S}(m)$. Now suppose for a contradiction that $m = \sum_{j=1}^i a_j$ for $i \in \nn_{\ge 2}$, where $a_1, \dots, a_i \in \mathcal{A}(P_S)$. Take $j \in \nn_{\ge K}$ such that $a_1, \dots, a_i \in \mathcal{A}(P_j)$. Then the way in which $P_{j+1}$ was constructed ensures that $|\mathsf{Z}_{P_{j+1}}(a_1 + a_2)| \ge 2$ and, therefore, $|\mathsf{Z}_{P_{j+1}}(m)| \ge 2$. As $\mathsf{Z}_{P_{j+1}}(m) \subseteq \mathsf{Z}_{P_S}(m)$, it follows that $|\mathsf{Z}_{P_S}(m)| \ge 2$, which contradicts that $m$ is a molecule. Hence $\mathcal{M}(P_S) = \mathcal{A}(P_S)$.
	
	Finally, we argue that the monoids constructed are not isomorphic. Let $S$ and $S'$ be two distinct members of the collection $\mathcal{S}$ and suppose, by way of contradiction, that $\psi \colon P_S \to P_{S'}$ is a monoid isomorphism. Because the only homomorphisms of Puiseux monoids are given by rational multiplication, there exists $q \in \qq_{> 0}$ such that $P_{S'} = q P_S$. In this case, all but finitely many primes in $\mathsf{D}_\pp(P_S)$ belong to $\mathsf{D}_\pp(P_{S'})$. Since $\mathsf{D}_\pp(P_S) \cap \mathsf{D}_\pp(P_{S'}) = \emptyset$ when $S \neq S'$, we get a contradiction.
\end{proof}
\medskip

\subsection{Molecules of Prime Reciprocal Monoids} For the remaining of this section, we focus our attention on the class consisting of all prime reciprocal monoids.

\begin{definition}
	Let $S$ be a nonempty set of primes. A Puiseux monoid $P$ is \emph{prime reciprocal over} $S$ if there exists a set of positive rationals $R$ such that $P = \langle R \rangle$, $\mathsf{d}(R) = S$, and $\mathsf{d}(r) = \mathsf{d}(r')$ implies $r = r'$ for all $r,r' \in R$.
\end{definition}

Within the scope of this paper, the term \emph{prime reciprocal monoid} refers to a Puiseux monoid that is prime reciprocal over some nonempty set of primes. Let us remark that if a Puiseux monoid $P$ is prime reciprocal, then there exists a \emph{unique} $S \subseteq \pp$ such that $P$ is prime reciprocal over $S$. It is easy to verify that every prime reciprocal Puiseux monoid is atomic.

\begin{prop}[cf. Theorem~\ref{thm:more molecules than atoms in any NS}(1)] \label{prop:PPM has infinitely many molecules that are not atoms}
	There exist infinitely many non-finitely generated atomic Puiseux monoids $P$ such that $|\mathcal{M}(P) \! \setminus \! \mathcal{A}(P)| = \infty$.
\end{prop}

\begin{proof}
	As in the proof of Theorem~\ref{thm:F(M) equals A(M)}, let $\mathcal{S} = \{S_n : n \in \nn\}$ be a collection of infinite and pairwise-disjoint subsets of $\pp \setminus \{2\}$. For every $n \in \nn$, let $P_n$ be a prime reciprocal Puiseux monoid over $S_n$. Fix $a \in \mathcal{A}(P_n)$, and take a factorization
	\[	
		z := \sum_{i=1}^k \alpha_i a_i  \in \mathsf{Z}(2a)
	\]
	for some $k \in \nn$, pairwise distinct atoms $a_1, \dots, a_k$, and $\alpha_1,\dots,\alpha_k \in \nn_0$. Since $\mathsf{d}(a) \neq 2$, after applying the $\mathsf{d}(a)$-adic valuation on both sides of the equality $2a = \sum_{i=1}^t \alpha_i a_i$, one obtains that $z = 2a$. So $2a \in \mathcal{M}(P_n) \setminus \mathcal{A}(P_n)$ and, as a result, $|\mathcal{M}(P_n) \setminus \mathcal{A}(P_n)| = \infty$. Now suppose, by way of contradiction, that $P_i \cong P_j$ for some $i,j \in \nn$ with $i \neq j$. Since isomorphisms of Puiseux monoids are given by rational multiplication, there exists $q \in \qq_{> 0}$ such that $P_j = q P_i$. However, this implies that only finitely many primes in $\mathsf{d}(P_i)$ are not contained in $\mathsf{d}(P_j)$, which contradicts that $S_i \cap S_j = \emptyset$. Hence no two monoids in $\{P_n : n \in \nn\}$ are isomorphic, and the proposition follows.
\end{proof}

Theorem~\ref{thm:F(M) equals A(M)} and Proposition~\ref{prop:PPM has infinitely many molecules that are not atoms} guarantee the existence of infinitely many non-finitely generated atomic Puiseux monoids~$P$ and~$Q$ with $|\mathcal{M}(P) \! \setminus \! \mathcal{A}(P)| = 0$ and $|\mathcal{M}(Q) \! \setminus \! \mathcal{A}(Q)| = \infty$. 

\begin{conj} (cf. Conjecture~\ref{conj:molecules minus atoms in NS}) \label{q:molecules minus atoms in PM}
	For every $n \in \nn$ there exists a non-finitely generated atomic Puiseux monoid $P$ satisfying that $|\mathcal{M}(P) \! \setminus \! \mathcal{A}(P)| = n$.
\end{conj}
\medskip

Before characterizing the molecules of prime reciprocal monoids, let us introduce the concept of maximal multiplicity. Let $P$ be a Puiseux monoid. For $x \in P$ and $a \in \mathcal{A}(P)$ we define the \emph{maximal multiplicity} of $a$ in $x$ to be
\[
	\mathsf{m}(a,x) := \max\{n \in \nn_0 : na \mid_P x\}.
\]

\begin{prop} \label{prop:sufficient condition for molecules in PPM}
	Let $P$ be a prime reciprocal monoid, and let $x \in P$. If $\mathsf{m}(a,x) < \mathsf{d}(a)$ for all $a \in \mathcal{A}(P)$, then $x \in \mathcal{M}(P)$.
\end{prop}

\begin{proof}
	Suppose, by way of contradiction, that $x \notin \mathcal{M}(P)$. Then there exist $k \in \nn$, elements $\alpha_i, \beta_i \in \nn_0$ (for $i=1,\dots, k$), and pairwise distinct atoms $a_1, \dots, a_k$ such that
	\[
		z:= \sum_{i=1}^k \alpha_i a_i \ \text{ and } \ z' := \sum_{i=1}^k \beta_i a_i
	\]
	are two distinct factorizations in $\mathsf{Z}(x)$. As $z \neq z'$, there is an index $i \in \{1,\dots,k\}$ such that $\alpha_i \neq \beta_i$. Now we can apply the $\mathsf{d}(a_i)$-adic valuation to both sides of the equality
	\[
		\sum_{i=1}^k \alpha_i a_i = \sum_{i=1}^k \beta_i a_i
	\]
	to verify that $\mathsf{d}(a_i) \mid \beta_i - \alpha_i$. As $\alpha_i \neq \beta_i$, we obtain that
	\[
		\mathsf{m}(a_i, x) \ge \max\{\alpha_i, \beta_i\} \ge \mathsf{d}(a_i).
	\]
	However, this contradicts the fact that $\mathsf{m}(a,x) < \mathsf{d}(a)$ for all $a \in \mathcal{A}(P)$. As a consequence, $x \in \mathcal{M}(P)$.
\end{proof}

For $S \subseteq \pp$, we call the monoid $E_S := \langle 1/p : p \in S \rangle$ the \emph{elementary} prime reciprocal monoid over $S$; if $S = \pp$ we say that $E_S$ is \emph{the} elementary prime reciprocal monoid. It was proved in \cite[Section~5]{GG18} that every submonoid of the elementary prime reciprocal monoid is atomic. This gives a large class of non-finitely generated atomic Puiseux monoids, which contains each prime reciprocal monoid.

\begin{prop} \label{prop:factorial equivalence for elementary primary monoids}
	Let $S$ be an infinite set of primes, and let $E_S$ be the elementary prime reciprocal monoid over $S$. For $x \in E_S$, the following conditions are equivalent:
	\begin{enumerate}
		\item $x \in \mathcal{M}(E_S)$;
		\vspace{3pt}
		\item $1$ does not divide $x$ in $E_S$;
		\vspace{3pt}
		\item$\mathsf{m}(a,x) < \mathsf{d}(a)$ for all $a \in \mathcal{A}(E_S)$;
		\vspace{3pt}
		\item If $a_1, \dots, a_n \in \mathcal{A}(E_S)$ are distinct atoms and $\alpha_1, \dots, \alpha_n \in \nn_0$ satisfy that $\sum_{j=1}^n \alpha_j a_j \in \mathsf{Z}(x)$, then $\alpha_j < \mathsf{d}(a_j)$ for each $j = 1,\dots,n$.
	\end{enumerate}
\end{prop}

\begin{proof}
	First, let us recall that since $E_S$ is atomic, $\mathcal{M}(E_S)$ is divisor-closed. On the other hand, note that for any two distinct atoms $a,a' \in \mathcal{A}(E_S)$, both factorizations $\mathsf{d}(a) \, a$ and $\mathsf{d}(a') \, a'$ are in $\mathsf{Z}(1)$. Therefore $1 \notin \mathcal{M}(E_S)$. Because the set of molecules of $E_S$ is divisor-closed, $1 \nmid_{E_S} m$ for any $m \in \mathcal{M}(E_S)$; in particular, $1 \nmid_{E_S} x$. Thus, (1) implies (2). If $\mathsf{m}(a,x) \ge \mathsf{d}(a)$ for $a \in \mathcal{A}(E_S)$, then
	\[
		x = \mathsf{m}(a,x) \, a + y = 1 + (\mathsf{m}(a,x) - \mathsf{d}(a)) \, a + y
	\]
	for some $y \in E_S$. As a result, $1 \mid_{E_S} x$, from which we can conclude that~(2) implies (3). It is obvious that (3) and (4) are equivalent conditions. Finally, the fact that (3) implies~(1) follows from Proposition~\ref{prop:sufficient condition for molecules in PPM}.
\end{proof}

\begin{cor}
	Let $S$ be an infinite set of primes, and let $E_S$ be the elementary prime reciprocal monoid over $S$. Then $|\mathsf{Z}(x)| = \infty$ for all $x \notin \mathcal{M}(E_S)$.
\end{cor}

In order to describe the set of molecules of an arbitrary prime reciprocal monoid, we need to cast its atoms into two categories.

\begin{definition}
	Let $P$ be a prime reciprocal monoid. We say that $a \in \mathcal{A}(P)$ is \emph{stable} if the set $\{a' \in \mathcal{A}(P) : \mathsf{n}(a') = \mathsf{n}(a)\}$ is infinite, otherwise we say that $a$ is \emph{unstable}. If every atom of $P$ is stable (resp., unstable), then we call $P$ \emph{stable} (resp., \emph{unstable}).
\end{definition}

For a prime reciprocal monoid $P$, we let $\mathcal{S}(P)$ denote the submonoid of $P$ generated by the set of stable atoms. Similarly, we let $\mathcal{U}(P)$ denote the submonoid of $P$ generated by the set of unstable atoms. Clearly, $P$ is stable (resp., unstable) if and only if $P = \mathcal{S}(P)$ (resp., $P = \mathcal{U}(P)$). In addition, $P = \mathcal{S}(P) + \mathcal{U}(P)$, and $\mathcal{S}(P) \cap \mathcal{U}(P)$ is trivial only when either $\mathcal{S}(P)$ or $\mathcal{U}(P)$ is trivial. Clearly, if $P$ is stable, then it cannot be finitely generated. Finally, we say that $u \in \mathcal{U}(P)$ is \emph{absolutely unstable} provided that $u$ is not divisible by any stable atom in $P$, and we let $\mathcal{U}^a(P)$ denote the set of all absolutely unstable elements of~$P$.

\begin{example}
	Let $\{p_n\}$ be the strictly increasing sequence with underlying set $\pp \setminus \{2\}$, and consider the prime reciprocal monoid $P$ defined as
	\[
		P := \bigg\langle \frac{3 + (-1)^n}{p_{2n-1}}, \frac{p_{2n} - 1}{p_{2n}} \ : \ n \in \nn \bigg\rangle.
	\]
	Set $a_n = \frac{3 + (-1)^n}{p_{2n-1}}$ and $b_n =  \frac{p_{2n} - 1}{p_{2n}}$. One can readily verify that $P$ is an atomic monoid with $\mathcal{A}(P) = \{a_n, b_n : n \in \nn\}$. As both sets
	\[
		\{n \in \nn : \mathsf{n}(a_n) = 2\} \quad \text{ and } \quad \{n \in \nn : \mathsf{n}(a_n) = 4\}
	\]
	have infinite cardinality, $a_n$ is a stable atom for every $n \in \nn$. In addition, since $\{\mathsf{n}(b_n)\}$ is a strictly increasing sequence bounded below by $\mathsf{n}(b_1) = 4$ and $\mathsf{n}(a_n) \in \{2,4\}$, the element $b_n$ is an unstable atom for every $n \in \nn_{\ge 2}$. Also, notice that $4/3 = 2a_1 \in \mathcal{S}(P)$, but $4/3 \notin \mathcal{U}(P)$ because $\mathsf{d}(4/3) = 3 \notin \mathsf{d}(\mathcal{U}(P))$. Furthermore, for every $n \in \nn$ the element $u_n := (p_{2n} - 1) b_n \in \mathcal{U}(P)$ is not in $\mathcal{S}(P)$ because $p_{2n} = \mathsf{d}(u_n) \notin \mathsf{d}(\mathcal{S}(P))$. However, $\mathcal{S}(P) \cap \mathcal{U}(P) \ne \emptyset$ since the element $4 = 6 a_1 = 5 b_1$ belongs to both $\mathcal{S}(P)$ and $\mathcal{U}(P)$. Finally, we claim that $2b_n$ is absolutely unstable for every $n \in \nn$. If this were not the case, then $2b_k \notin \mathcal{M}(P)$ for some $k \in \nn$. By Proposition~\ref{prop:sufficient condition for molecules in PPM} there exists $a \in \mathcal{A}(P)$ such that $\mathsf{m}(a, 2b_k) \ge \mathsf{d}(a)$. In this case, one would obtain that $2b_k \ge \mathsf{m}(a,2b_k)a \ge \mathsf{d}(a)a = \mathsf{n}(a) \ge 2$, contradicting that $b_n < 1$ for every $n \in \nn$. Thus, $2b_n \in \mathcal{U}^a(P)$ for every $n \in \nn$.
\end{example}

\begin{prop} \label{prop: characterization of factorial elements in stable PPM}
	Let $P$ be a prime reciprocal monoid that is stable, and let $x \in P$. Then $x \in \mathcal{M}(P)$ if and only if~$\mathsf{n}(a)$ does not divide $x$ in $P$ for any $a \in \mathcal{A}(P)$.
\end{prop}

\begin{proof}
	For the direct implication, assume that $x \in \mathcal{M}(P)$ and suppose, by way of contradiction, that $\mathsf{n}(a) \mid_P x$ for some $a \in \mathcal{A}(P)$. Since $a$ is a stable atom, there exist $p_1, p_2 \in \pp$ with $p_1 \neq p_2$ such that $\gcd(p_1 p_2, \mathsf{n}(a)) = 1$ and $\mathsf{n}(a)/p_1, \mathsf{n}(a)/p_2 \in \mathcal{A}(P)$. As $\mathsf{n}(a) \mid_P x$, we can take $a_1, \dots, a_k \in \mathcal{A}(P)$ such that $x = \mathsf{n}(a) + a_1 + \dots + a_k$. Therefore
	\[
		p_1 \frac{\mathsf{n}(a)}{p_1} + a_1 + \dots + a_k \ \text{ and } \ p_2 \frac{\mathsf{n}(a)}{p_2} + a_1 + \dots + a_k
	\]
	are two distinct factorizations in $\mathsf{Z}(x)$, contradicting that $x$ is a molecule.
	Conversely, suppose that $x$ is not a molecule. Consider two distinct factorizations $z := \sum_{i=1}^k \alpha_i a_i$ and $z' := \sum_{i=1}^k \beta_i a_i$ in $\mathsf{Z}(x)$, where $k \in \nn$, $\alpha_i, \beta_i \in \nn_0$, and $a_1, \dots, a_k \in \mathcal{A}(P)$ are pairwise distinct atoms. Pick an index $j \in \{1, \dots, k\}$ such that $\alpha_j \neq \beta_j$ and assume, without loss of generality, that $\alpha_j < \beta_j$. After applying the $\mathsf{d}(a_j)$-adic valuations on both sides of the equality
	\[
		\sum_{i=1}^k \alpha_i a_i = \sum_{i=1}^k \beta_i a_i
	\]
	one finds that the prime $\mathsf{d}(a_j)$ divides $\beta_j - \alpha_j$. Therefore $\beta_j > \mathsf{d}(a_j)$ and so
	\[
		x = \mathsf{n}(a_j) + (\beta_j - \mathsf{d}(a_j))a_j + \sum_{i \neq j} \alpha_i a_i.
	\vspace{-6pt}
	\]
	Hence $\mathsf{n}(a_j) \mid_P x$, which concludes the proof.
\end{proof}

Observe that the reverse implication of Proposition~\ref{prop: characterization of factorial elements in stable PPM} does not require that the equality $\mathcal{S}(P) = P$ holds. However, the stability of $P$ is required for the direct implication to hold as the following example illustrates.

\begin{example}
	Let $\{p_n\}$ be the strictly increasing sequence with underlying set $\pp \setminus \{2\}$, and consider the unstable prime reciprocal monoid
	\[
		P := \bigg\langle \frac 12, \frac{p_n^2 -1}{p_n} \, : \, n \in \nn \bigg\rangle.
	\]
	Because the smallest two atoms of $P$ are $1/2$ and $8/3$, it immediately follows that $m := 2(1/2) + 8/3 \notin \langle 1/2 \rangle$ must be a molecule of $P$. In addition, notice that $1 = \mathsf{n}(1/2)$ divides $m$ in $P$.
\end{example}

We conclude this section characterizing the molecules of prime reciprocal monoids.

\begin{theorem} \label{thm:characterization of factorial elements in PPM}
	Let $P$ be a prime reciprocal monoid. Then $x \in P$ is a molecule if and only if $x = s + u$ for some $s \in \mathcal{S}(P) \cap \mathcal{M}(P)$ and $u \in \mathcal{U}^a(P) \cap \mathcal{M}(P)$.
\end{theorem}

\begin{proof}
	First, suppose that $x$ is a molecule. As $P = \mathcal{S}(P) + \mathcal{U}(P)$, there exist $s \in \mathcal{S}(P)$ and $u \in \mathcal{U}(P)$ such that $x = s + u$. The fact that $x \in \mathcal{M}(P)$ guarantees that $s,u \in \mathcal{M}(P)$. On the other hand, since $|\mathsf{Z}(u)| = 1$ and $u$ can be factored using only unstable atoms, $u$ cannot be divisible by any stable atom in $P$. Thus, $u \in \mathcal{U}^a(P)$, and the direct implication follows. 
	
	For the reverse implication, assume that $x = s + u$, where $s \in \mathcal{S}(P) \cap \mathcal{M}(P)$ and $u \in \mathcal{U}^a(P) \cap \mathcal{M}(P)$. We first check that $x$ can be uniquely expressed as a sum of two elements $s$ and $u$ contained in the sets $\mathcal{S}(P) \cap \mathcal{M}(P)$ and $\mathcal{U}^a(P) \cap \mathcal{M}(P)$, respectively. To do this, suppose that $x = s + u = s' + u'$, where $s' \in \mathcal{S}(P) \cap \mathcal{M}(P)$ and $u' \in \mathcal{U}^a(P) \cap \mathcal{M}(P)$. Take pairwise distinct stable atoms $a_1, \dots, a_k$ of $P$ for some $k \in \nn$ such that $z = \sum_{i=1}^k \alpha_i a_i \in \mathsf{Z}_P(s)$ and $z' = \sum_{i=1}^k \alpha'_i a_i \in \mathsf{Z}_P(s')$, where $\alpha_j, \alpha'_j \in \nn_0$ for $j = 1, \dots, k$. Because $u$ and $u'$ are absolutely unstable elements, they are not divisible in $P$ by any of the atoms $a_i$'s. Thus, $\mathsf{d}(a_j) \nmid \mathsf{d}(u)$ and $\mathsf{d}(a_j) \nmid \mathsf{d}(u')$ for any $j \in \{1,\dots, k\}$. Now for each $j = 1,\dots,k$ we can apply the $\mathsf{d}(a_j)$-adic valuation in both sides of the equality
	\[
		u + \sum_{i=1}^k \alpha_i a_i  = u' + \sum_{i=1}^k \alpha'_i a_i 
	\]
	to conclude that the prime $\mathsf{d}(a_j)$ must divide $\alpha_j - \alpha'_j$. Therefore either $z = z'$ or there exists $j \in \{1,\dots,k\}$ such that $|\alpha_j - \alpha'_j| > \mathsf{d}(a_j)$. Suppose that $|\alpha_j - \alpha'_j| > \mathsf{d}(a_j)$ for some $j$, and say $\alpha_j > \alpha'_j$. As $\alpha_j > \mathsf{d}(a_j)$, one can replace $\alpha_j a_j$ by $(\alpha_j - \mathsf{d}(a_j))a_j + \mathsf{n}(a_j)$ in $s = \phi(z) = \alpha_1 a_1 + \dots + \alpha_k a_k$ to find that $\mathsf{n}(a_j)$ divides $s$ in $\mathcal{S}(P)$, which contradicts Proposition~\ref{prop: characterization of factorial elements in stable PPM}. Then $z = z'$. Therefore $s' = s$ and $u' = u$.
	
	Finally, we argue that $x \in \mathcal{M}(P)$. Write $x = \sum_{i=1}^{\ell} \gamma_i a_i + \sum_{i=1}^\ell \beta_i b_i$ for $\ell \in \nn_{\ge k}$, pairwise distinct stable atoms $a_1, \dots, a_\ell$ (where $a_1, \dots, a_k$ are the atoms showing up in $z$), pairwise distinct unstable atoms $b_1, \dots, b_\ell$, and coefficients $\gamma_i, \beta_i \in \nn_0$ for every $i=1,\dots,\ell$. Set $z''' := \sum_{i=1}^{\ell} \gamma_i a_i$ and $w''' = \sum_{i=1}^\ell \beta_i b_i$. Note that, \emph{a priori}, $\phi(z''')$ and $\phi(w''')$ are not necessarily molecules.  As in the previous paragraph, we can apply $\mathsf{d}(a_j)$-adic valuation to both sides of the equality
	\[
		u + \sum_{i=1}^k \alpha_i a_i = \sum_{i=1}^{\ell} \gamma_i a_i + \sum_{i=1}^\ell \beta_i b_i
	\]
	to find that $z''' = z$. Hence $\phi(z''') = s$ and $\phi(w''') = u$ are both molecules. Therefore $z'''$ must be the unique factorization of $s$, while $w'''$ must be the unique factorization of $u$. As a result, $x \in \mathcal{M}(P)$.
\end{proof}
\medskip

\section{Molecules of Puiseux Algebras}
\label{sec:molecules of PA}

Let $M$ be a monoid and let $R$ be a commutative ring with identity. Then $R[X;M]$ denotes the ring of all functions $f \colon M \to R$ having finite \emph{support}, which means that $\supp(f) := \{s \in M : f(s) \neq 0 \}$ is finite. We represent an element $f \in R[X;M]$ by
\[
	f(X) = \sum_{i=1}^n f(s_i)X^{s_i},
\]
where $s_1, \dots, s_n$ are the elements in $\supp(f)$. The ring $R[X;M]$ is called the \emph{monoid ring} of $M$ \emph{over}~$R$, and the monoid $M$ is called the \emph{exponent monoid} of $R[X;M]$. For a field $F$, we will say that $F[X;M]$ is a \emph{monoid algebra}. As we are primarily interested in the molecules of monoid algebras of Puiseux monoids, we introduce the following definition.

\begin{definition}
	If $F$ is a field and $P$ is a Puiseux monoid, then we say that $F[X;P]$ is a \emph{Puiseux algebra}. If $N$ is a numerical semigroup, then $F[X;N]$ is called a \emph{numerical semigroup algebra}.
\end{definition}

Let $F[X;P]$ be a Puiseux algebra. We write any element $f \in F[X;P] \setminus \{0\}$ in \emph{canonical representation}, that is, $f(X) = \alpha_1 X^{q_1} + \dots + \alpha_k X^{q_k}$ with $\alpha_i \neq 0$ for every $i = 1, \dots, k$ and $q_1 > \dots > q_k$. It is clear that any element of $F[X;P] \setminus \{0\}$ has a unique canonical representation. In this case, $\deg(f) := q_1$ is called the \emph{degree} of $f$, and we obtain that the degree identity $\deg(fg) = \deg(f) + \deg(g)$ holds for all $f, g \in F[X;P] \setminus \{0\}$. As for polynomials, we say that $f$ is a \emph{monomial} if $k = 1$. It is not hard to verify that $F[X;P]$ is an integral domain with group of units $F^\times$, although this follows from~\cite[Theorem~8.1]{rG84} and~\cite[Theorem~11.1]{rG84}. Finally, note that, unless $P \cong (\nn_0,+)$, no monomial of $F[X;P]$ can be a prime element; this is a consequence of the trivial fact that non-cyclic Puiseux monoids do not contain prime elements.

For an integral domain $R$, we let $R_{\text{red}}$ denote the reduced monoid of the multiplicative monoid of $R$.

\begin{definition}
	Let $R$ be an integral domain. We call a nonzero non-unit $r \in R$ a \emph{molecule} if $rR^\times$ is a molecule of $R_{\text{red}}$.
\end{definition}

Let $R$ be an integral domain. By simplicity, we let $\mathcal{A}(R)$, $\mathcal{M}(R)$, $\mathsf{Z}(R)$, and $\phi_R$ denote $\mathcal{A}(R_{\text{red}})$, $\mathcal{M}(R_{\text{red}})$, $\mathsf{Z}(R_{\text{red}})$, and $\phi_{R_{\text{\text{red}}}}$, respectively. In addition, for a nonzero non-unit $r \in R$, we let $\mathsf{Z}_R(r)$ and $\mathsf{L}_R(r)$ denote $\mathsf{Z}_{R_{\text{red}}}(rR^\times)$ and $\mathsf{L}_{R_{\text{red}}}(rR^\times)$, respectively.

\begin{prop} \label{prop:monoid molecules}
	Let $F$ be a field, and let $P$ be a Puiseux monoid. For a nonzero $\alpha \in F$, a monomial $X^q \in \mathcal{M}(F[X;P])$ if and only if $q \in \mathcal{M}(P)$.
\end{prop}

\begin{proof}
	Consider the canonical monoid monomorphism $\mu \colon P \to F[X;P] \setminus \{0\}$ given by $\mu(q) = X^q$. It follows from~\cite[Lemma~3.1]{CM11} that an element $a \in P$ is an atom if and only if the monomial $X^a$ is irreducible in $F[X;P]$ (or, equivalently, an atom in the reduced multiplicative monoid of $F[X;P]$). Therefore $\mu$ lifts canonically to the monomorphism $\bar{\mu} \colon \mathsf{Z}(P) \to \mathsf{Z}(F[X;P])$ determined by the assignments $a \mapsto X^a$ for each $a \in \mathcal{A}(P)$, preserving not only atoms but also factorizations of the same element. Put formally, this means that the diagram
	\[
		\begin{CD}
		\mathsf{Z}(P)      		 @>\bar{\mu}>>       \mathsf{Z}(F[X;P])        \\
		@V\phi_PVV                							 @V\phi_{F[X;P]}VV    	   \\
		P    							@>\mu>>          	  F[X;P]_{\text{red}}         \\     
		\end{CD}
	\]
	commutes, and the (fiber) restriction maps $\bar{\mu}_q \colon \mathsf{Z}_P(q) \to \mathsf{Z}_{F[X;P]}(X^q)$ of $\bar{\mu}$ are bijections for every $q \in P$. Hence $|\mathsf{Z}_P(q)| = 1$ if and only if $|\mathsf{Z}_{F[X;P]}(X^q)| = 1$ for all $q \in P^\bullet$, which concludes our proof.
\end{proof}

\begin{cor} \label{cor:PA with infinitely more monomial molecules than atoms}
	For each field $F$, there exists an atomic Puiseux monoid $P$ whose Puiseux algebra satisfies that $|\mathcal{M}(F[X;P]) \setminus \mathcal{A}(F[X;P])| = \infty$.
\end{cor}

\begin{proof}
	It is an immediate consequence of Proposition~\ref{prop:PPM has infinitely many molecules that are not atoms} and Proposition~\ref{prop:monoid molecules}.
\end{proof}

The \emph{difference group} $\gp(M)$ of a monoid $M$ is the abelian group (unique up to isomorphism) satisfying that any abelian group containing a homomorphic image of $M$ will also contain a homomorphic image of $\gp(M)$. An element $x \in \gp(M)$ is called a \emph{root element} of $M$ if $nx \in M$ for some $n \in \nn$. The subset $\widetilde{M}$ of $\gp(M)$ consisting of all root elements of $M$ is called the \emph{root closure} of $M$. If $\widetilde{M} = M$, then $M$ is called \emph{root-closed}. From now on, we assume that each Puiseux monoid $P$ we mention here is root-closed. Before providing a characterization for the irreducible elements of $F[X;P]$, let us argue the following two easy lemmas.

\begin{lemma} \label{lem:denominator sets of PM are closed under LCM}
	Let $P$ be a Puiseux monoid. Then $\mathsf{d}(P^\bullet)$ is closed under taking least common multiples.
\end{lemma}

\begin{proof}
	Take $d_1, d_2 \in \mathsf{d}(P^\bullet)$ and $q_1, q_2 \in P^\bullet$ with $\mathsf{d}(q_1) = d_1$ and $\mathsf{d}(q_2) = d_2$. Now set $d = \gcd(d_1, d_2)$ and $n = \gcd(\mathsf{n}(q_1), \mathsf{n}(q_2))$. It is clear that $n$ is the greatest common divisor of $(d_2/d) \mathsf{n}(q_1)$ and $(d_1/d) \mathsf{n}(q_2)$. So there exist $m \in \nn$ and $c_1, c_2 \in \nn_0$ such that
	\begin{equation} \label{eq:gcd}
		n \big(1 + m \, \lcm(d_1, d_2) \big) = c_1 \frac{d_2}{d} \mathsf{n}(q_1) + c_2 \frac{d_1}{d} \mathsf{n}(q_2).
	\end{equation}
	Using the fact that $d \, \lcm(d_1,d_2) = d_1 d_2$, one obtains that
	\[
		\frac{n \big(1 + m \, \lcm(d_1, d_2) \big)}{\lcm(d_1,d_2)} = c_1 q_1 + c_2 q_2 \in P
	\]
	after dividing both sides of the equality~(\ref{eq:gcd}) by $\lcm(d_1,d_2)$. In addition, note that $n (1 + m \, \lcm(d_1, d_2) )$ and $\lcm(d_1,d_2)$ are relatively prime. Hence $\lcm(d_1, d_2) \in \mathsf{d}(P^\bullet)$, from which the lemma follows.
\end{proof}

\begin{lemma} \label{lem:denominator division in a root-closed monoid}
	Let $P$ be a root-closed Puiseux monoid containing $1$. Then $1/d \in P$ for all $d \in \mathsf{d}(P^\bullet)$.
\end{lemma}

\begin{proof}
	Let $d \in \mathsf{d}(P^\bullet)$, and take $r \in P^\bullet$ such that $\mathsf{d}(r) = d$. As $\gcd(\mathsf{n}(r), \mathsf{d}(r)) = 1$, there exist $a,b \in \nn_0$ such that $a \mathsf{n}(r) - b \, \mathsf{d}(r) = 1$. Therefore
	\[
		\frac{1}{d} = \frac{a \mathsf{n}(r) - b \, \mathsf{d}(r)}{d} = a r - b \in \gp(P).
	\]
	This, along with the fact that $d (1/d) = 1 \in P$, ensures that $1/d$ is a root element of $P$. Since $P$ is root-closed, it must contain $1/d$, which concludes our argument.
\end{proof}

We are in a position now to characterize the irreducibles of $F[X;P]$.

\begin{prop} \label{prop:irreducibles in a PA of a root closed PM}
	Let $F$ be a field, and let $P$ be a root-closed Puiseux monoid containing $1$. Then $f \in F[X;P] \setminus F$ is irreducible in $F[X;P]$ if and only if $f(X^m)$ is irreducible in $F[X]$ for every $m \in \mathsf{d}(P^\bullet)$ that is a common multiple of the elements of $\mathsf{d}(\supp(f))$.
\end{prop}

\begin{proof}
	Suppose first that $f \in F[X;P] \setminus F$ is an irreducible element of $F[X;P]$, and let $m \in \mathsf{d}(P^\bullet)$ be a common multiple of the elements of $\mathsf{d}\big(\supp(f) \big)$. Then $f(X^m)$ is an element of $F[X]$. Take $g, h \in F[X]$ such that $f(X^m) = g(X) \, h(X)$. As $P$ is a root-closed and $m \in \mathsf{d}(P^\bullet)$, Lemma~\ref{lem:denominator division in a root-closed monoid} ensures that $g(X^{1/m}), h(X^{1/m}) \in F[X;P]$. Thus, $f(X) = g(X^{1/m}) h(X^{1/m})$ in $F[X;P]$. Since $f$ is irreducible in $F[X;P]$ either $g(X^{1/m}) \in F$ or $h(X^{1/m}) \in F$, which implies that either $g \in F$ or $h \in F$. Hence $f(X^m)$ is irreducible in $F[X]$.
	
	Conversely, suppose that $f \in F[X;P]$ satisfies that $f(X^m)$ is an irreducible polynomial in $F[X]$ for every $m \in \mathsf{d}(P^\bullet)$ that is a common multiple of the elements of the set $\mathsf{d}(\supp(f))$. To argue that $f$ is irreducible in $F[X;P]$ suppose that $f = g \, h$ for some $g, h \in F[X;P]$. Let $m_0$ be the least common multiple of the elements of $\mathsf{d}(\supp(g)) \cup \mathsf{d}(\supp(h))$. Lemma~\ref{lem:denominator sets of PM are closed under LCM} guarantees that $m_0 \in \mathsf{d}(P^\bullet)$. Moreover, $f = g \, h$ implies that $m_0$ is a common multiple of the elements of $\mathsf{d}(\supp(f))$. As a result, the equality $f(X^{m_0}) = g(X^{m_0})h(X^{m_0})$ holds in $F[X]$. Since $f(X^{m_0})$ is irreducible in $F[X]$, either $g(X^{m_0}) \in F$ or $h(X^{m_0}) \in F$ and, therefore, either $g \in F$ or $h \in F$. This implies that $f$ is irreducible in $F[X;P]$, as desired.
\end{proof}

We proceed to show the main result of this section.

\begin{theorem} \label{thm:U-UFD Puiseux algebras}
	Let $F$ be a field, and let $P$ be a root-closed Puiseux monoid. Hence
	\[
		\mathcal{M}(F[X;P]) = \langle \mathcal{A}(F[X;P]) \rangle.
	\]
\end{theorem}

\begin{proof}
	As each molecule of $F[X;P]$ is a product of irreducible elements in $F[X;P]$, the inclusion $\mathcal{M}(F[X;P]) \subseteq \langle \mathcal{A}(F[X;P]) \rangle$ holds trivially. For the reverse inclusion, suppose that $f \in F[X;P] \setminus F$ can be written as a product of irreducible elements in $F[X;P]$. As a result, there exist $k,\ell \in \nn$ and irreducible elements $g_1, \dots, g_k$ and $h_1, \dots, h_\ell$ in $F[X;P]$ satisfying that
	\begin{equation} \label{eq:two products of irreducibles}
	g_1(X) \cdots g_k(X) = f(X) = h_1(X) \cdots h_\ell(X).
	\end{equation}
	Let $m$ be the least common multiple of all the elements of the set
	\[
		\bigg(\bigcup_{i=1}^k\mathsf{d}\big(\supp(g_i)\big) \bigg) \bigcup  \bigg(\bigcup_{j=1}^\ell\mathsf{d}\big(\supp(h_j)\big) \bigg).
	\]
	Note that $f(X^m)$, $g_i(X^m)$ and $h_j(X^m)$ are polynomials in $F[X]$ for $i = 1, \dots, k$ and $j = 1, \dots, \ell$. Lemma~\ref{lem:denominator sets of PM are closed under LCM} ensures that $m \in \mathsf{d}(P^\bullet)$. On the other hand, $m$ is a common multiple of all the elements of $\mathsf{d}(\supp(g_i))$ (or all the elements of $\mathsf{d}(\supp(h_i))$). Therefore Proposition~\ref{prop:irreducibles in a PA of a root closed PM} guarantees that the polynomials $g_i(X^m)$ and $h_j(X^m)$ are irreducible in $F[X]$ for $i=1,\dots,k$ and $j=1,\dots,\ell$. After substituting $X$ by $X^m$ in~(\ref{eq:two products of irreducibles}) and using the fact that $F[X]$ is a UFD, one finds that $\ell = k$ and $g_i(X^m) = h_{\sigma(i)}(X^m)$ for some permutation $\sigma \in S_k$ and every $i = 1, \dots, k$. This, in turns, implies that $g_i = h_{\sigma(i)}$ for $i = 1, \dots, k$. Hence $|\mathsf{Z}_{F[X;P]}(f)| = 1$, which means that $f$ is a molecule of $F[X;P]$.
\end{proof}

As we have seen before, Corollary~\ref{cor:PA with infinitely more monomial molecules than atoms} guarantees the existence of a Puiseux algebra $F[X;P]$ satisfying that $|\mathcal{M}(F[X;P]) \setminus \mathcal{A}(F[X;P])| = \infty$. Now we use Theorem~\ref{thm:U-UFD Puiseux algebras} to construct an infinite class of Puiseux algebras satisfying a slightly more refined condition.

\begin{prop} \label{prop:U-UFD Puiseux algebras}
	For any field $F$, there exist infinitely many Puiseux monoids $P$ such that the algebra $F[X;P]$ contains infinite molecules that are neither atoms nor monomials.
\end{prop}

\begin{proof}
	Let $\{p_j\}$ be the strictly increasing sequence with underlying set $\pp$. Then for each $j \in \nn$ consider the Puiseux monoid $P_j = \langle 1/p^n_j \mid n \in \nn \rangle$. Fix $j \in \nn$, and take $P := P_j$. The fact that $\gp(P) = P \cup -P$ immediately implies that $P$ is a root-closed Puiseux monoid containing~$1$. Consider the Puiseux algebra $\qq[X;P]$ and the element $X + p \in \qq[X;P]$, where $p \in \pp$. To argue that $X + p$ is an irreducible element in $\qq[X;P]$, write $X + p = g(X) \, h(X)$ for some $g, h \in \qq[X;P]$. Now taking $m$ to be the maximum power of $p_j$ in the set $\mathsf{d}( \supp(g) \cup \supp(h) )$, one obtains that $X^m + p = g(X^m) \, h(X^m)$ in $\qq[X]$. Since $\qq[X]$ is a UFD, it follows by Eisenstein's criterion that $X^m + p$ is irreducible as a polynomial over $\qq$. Hence either $g(X) \in \qq$ or $h(X) \in \qq$, which implies that $X + p$ is irreducible in $\qq[X;P]$. Now it follows by Theorem~\ref{thm:U-UFD Puiseux algebras} that $(X + p)^n$ is a molecule in $\qq[X;P]$ for every $n \in \nn$. Clearly, the elements $(X + p)^n$ are neither atoms nor monomials.
	
	Finally, we prove that the algebras we have defined in the previous paragraph are pairwise non-isomorphic. To do so suppose, by way of contradiction, that $\qq[X;P_j]$ and $\qq[X;P_k]$ are isomorphic algebras for distinct $j,k \in \nn$. Let $\psi \colon \qq[X;P_j] \to \qq[X;P_k]$ be an algebra isomorphism. Since $\psi$ fixes $\qq$, it follows that $\psi(X^q) \notin \qq$ for any $q \in P_j^\bullet$. This implies that $\deg(\psi(X)) \in P_k^\bullet$. As $\mathsf{d}(P_j^\bullet)$ is unbounded there exists $n \in \nn$ such that $p_j^n > \mathsf{n}( \deg( \psi(X) ) )$. Observe that
	\begin{align} \label{eq: degrees 1}
		\deg \big( \psi(X) \big) = \deg \big( \psi \big( X^{\frac{1}{p_j^n}} \big)^{p_j^n} \big) = p_j^n \deg \big( \psi \big( X^{\frac{1}{p_j^n}} \big) \big).
	\end{align}
	Because $\gcd(p_j,d) = 1$ for every $d \in \mathsf{d}(P_k^\bullet)$, from~(\ref{eq: degrees 1}) one obtains that $p_j^n$ divides $\mathsf{n}( \deg \psi (X) )$, which contradicts that $p_j^n > \mathsf{n}( \deg( \psi(X) ) )$. Hence the Puiseux algebras in $\{P_j : j \in \nn\}$ are pairwise non-isomorphic, which completes our proof.
\end{proof}
\medskip

\section*{Acknowledgments}

\noindent While working on this paper, the first author was supported by the NSF-AGEP Fellowship and the UC Dissertation Year Fellowship. The authors would like to thank an anonymous referee, whose helpful suggestions help to improve the final version of this paper.

\end{document}